\documentclass[11pt,a4paper]{amsart}
\usepackage[pdftex]{graphicx}
\usepackage[pdftex]{xcolor}

\usepackage{mathtools}

\makeatletter
\newtheorem*{rep@theorem}{\rep@title}
\newcommand{\newreptheorem}[2]{%
\newenvironment{rep#1}[1]{%
 \def\rep@title{#2 \ref{##1}}%
 \begin{rep@theorem}}%
 {\end{rep@theorem}}}
\makeatother

\definecolor{RedOrange}{cmyk}{ 0, 0.77, 0.87, 0}
\definecolor{RoyalPurple}{cmyk}{ 0.84, 0.53, 0, 0}
\definecolor{YellowGreen}{cmyk}{ 0.44, 0, 0.74, 0}
\definecolor{Fuchsia}{cmyk}{ 0.47, 0.91, 0, 0.08}
\definecolor{Blue}{cmyk}{ 0.84, 0.53, 0, 0}
\definecolor{BlueViolet}{cmyk}{ 0.84, 0.53, 0, 0}
\definecolor{Black}{cmyk}{ 0.75, 0.68, 0.67, 0.9}
\usepackage{xcolor}
\usepackage{verbatim}
\usepackage{amsmath}
\usepackage{amssymb, mathrsfs}
\usepackage{amsbsy}

\usepackage{amscd}
\usepackage{amsthm}

\usepackage[english]{babel}
\usepackage{todonotes}
\usepackage[T1]{fontenc}

\usepackage{color}

\newcommand{\lf}{\lfloor}
\newcommand{\rf}{\rfloor}
\newcommand{\R}{\mathbb{R}}
\newcommand{\B}{\mathbb{B}}

\newcommand{\e}{\varepsilon}

\newcommand{\E}{\mathbb{E}}
\newcommand{\Z}{\mathbb{Z}}

\renewcommand{\P}{\mathbb{P}}
\renewcommand{\L}{\mathbb{L}}

\newcommand{\kA}{\mathcal{A}}
\newcommand{\kB}{\mathcal{B}}
\newcommand{\kC}{\mathcal{C}}

\newcommand{\kF}{\mathcal{F}}

\newcommand{\kE}{\mathcal{E}}

\newcommand{\rmT}{\mathrm{T}}
\newcommand{\rmB}{\mathrm{B}}

\newcommand{\lin}{\left[\kern-0.15em\left[}
\newcommand{\rin} {\right]\kern-0.15em\right]}
\newcommand{\linf}{[\kern-0.15em [}
\newcommand{\rinf} {]\kern-0.15em ]}
\newcommand{\ilin}{\left]\kern-0.15em\left]}
\newcommand{\irin} {\right[\kern-0.15em\right[}

\def\ben#1{\begin{equation}#1\end{equation}}

\def\al#1{\begin{align*}#1\end{align*}}
\def\aln#1{\begin{align}#1\end{align}}

\usepackage{constants}

\newconstantfamily{c}{symbol=c}

\newconstantfamily{a}{symbol=\alpha}

\newcommand{\secno}[1]{\thesection.\arabic{#1}}
\newconstantfamily{kE}{
symbol=\mathcal{E},
format=\secno,
reset={section}
}

\renewcommand{\tilde}{\widetilde}

\newtheorem{lem}{Lemma}[section]

\newtheorem{prop}[lem]{Proposition}
\newtheorem{thm}[lem]{Theorem}

\newtheorem {Def}[lem] {Definition}

\usepackage{color}
\definecolor{lilas}{RGB}{182, 102, 210}
\newcommand{\SN}{\color{blue}}

\numberwithin{equation}{section}

\def\ben#1{\begin{equation}#1\end{equation}}

\newcommand{\fr}{\frac}

\newcommand{\me}{\mathbf{e}}

\newcommand{\Le}{\left}
\newcommand{\Ri}{\right}

\title[Divergence of non-random fluctuation for Euclidean FPP]{Divergence of non-random fluctuation for Euclidean first-passage percolation}
\date{\today}
\author{Shuta Nakajima} 
\address[Shuta Nakajima]
{University of Basel, Basel, Switzerland.}
\email{shuta.nakajima@unibas.ch}


\begin{document}

\begin{abstract}

The non-random fluctuation is one of the central objects in first passage percolation. It was proved in \cite{N19} that for a particular asymptotic direction, it diverges in a lattice first passage percolation with an explicit lower bound. In this paper, we discuss the non-random fluctuation in Euclidean first passage percolations and show that it diverges in dimension $d\geq 2$ in this model also. Compared with the result in \cite{N19}, the present result is proved for any direction and improves the lower bound.

\end{abstract}

\maketitle

\section{Introduction}
First-passage percolation (FPP) was introduced by Hammersley and Welsh as a dynamical model of infection. One of the motivations of the studies on FPP is to understand the general behavior of subadditive processes.  To do this, a number of techniques and phenomena, such as Kingman's subadditive ergodic theorem and a sublinear variance, have been discovered and they have born fruitful results. See \cite{ADH} on the backgrounds and related topics.\\

We consider an Euclidean FPP on $\R^d$ with $d\geq{}2$, which is a variant of classical FPP and introduced in \cite{HN97}. The model is defined as follows. We consider a  Poisson point process $\Xi$ with Lebesgue intensity  on $\R^d$. We regard $\Xi$ as a subset of $\R^d$. For any $x\in\R^d$, we denote by $D(x)$ the closest point of $\Xi$ to $x$ with respect to the Euclidean norm $|\cdot|$. If there are multiple choices, we take one of them with a deterministic rule to break ties, though it does not happen almost surely.\\

A path $\gamma$ is a finite sequence of points $(x_0,\cdots,x_\ell)\subset\Xi$. Then we write $\gamma:x_0\to x_{\ell}$. We fix $\alpha>1$. Given a path $\gamma$, we define the passage time of $\gamma=(x_i)_{i=0}^\ell$ as
\ben{\label{definition of fpt}
{\rm T}(\gamma)=\sum_{i=1}^\ell |x_i-x_{i-1}|^\alpha,
}
where $|\cdot |$ is the Euclidean norm. For $x,y\in\R^d$, we define the {\em first passage time} between $x$ and $y$ as
$${\rm T}(x,y)=\inf_{\gamma:D(x)\to D(y)}{\rm T}(\gamma),$$
  where the infimum is taken over all finite paths $\gamma$ starting at $D(x)$ and ending at $D(y)$. It should be noted that if $\alpha\leq 1$, ${\rm T}(x,y)=|D(x)-D(y)|^\alpha$ , which is a rather trivial model. Hence we suppose $\alpha>1$. A path $\gamma$ from $D(x)$ to $D(y)$ is said to be {\em optimal} if it attains the first passage time  between $x$ and $y$, i.e. ${\rm T}(\gamma)={\rm T}(x,y)$. Note that for $x,y\in\R^d$, the optimal path between $x$ and $y$ is uniquely determined almost surely.\\

 One of the important property in our model is the so-called rotational invariance \cite[p.20]{CC18}. Indeed, for any rotation matrix, say $A$, $A\Xi=\{A x|\,\,x\in \Xi\}$ has the same distribution as $\Xi$. Hence, $({\rm T} (Ax,Ay))_{x,y\in\R^d}$ also has the same distribution as $({\rm T} (x,y))_{x,y\in\R^d}$.\\
  
  By Kingman's subadditive ergodic theorem, for any $x\in\R^d\backslash\{0\}$, there exists a non-random constant ${\rm g} \ge 0$ such that
\begin{equation}\label{kingman}
  {\rm g}=\lim_{t\to\infty}(t|x|)^{-1} {\rm T}(0,t x)=\lim_{t\to\infty}(t|x|)^{-1} \E[{\rm T}(0,t x)]\hspace{4mm}a.s.
\end{equation}
This ${\rm g}$, called the {\em time constant}, is independent of the choice of $x$ because of the rotational invariance. Moreover  it is known from \cite[Theorem~1]{HN97} that ${\rm g}$ is positive. Note that, since  ${\rm T}$ is subadditive (i.e., ${\rm T}(x,z)\leq {\rm T}(x,y)+{\rm T}(y,z)$), we have for $x\in\R^d$,
\ben{\label{time constant v.s. fpt}
{\rm g}|x| \le \E {\rm T}(0,x).
}
\subsection{Main results}
   We define
   $$\psi(t)={\rm Var}({\rm T}(0,t\mathbf{e}_1)),\,\phi(t)=\sqrt{\frac{t}{\psi(t)}},$$
where $(\mathbf{e}_i)_{i=1}^d$ is the canonical basis of $\R^d$. It  was proved in \cite{BDG18} that $\psi(t)\leq  \frac{C \,t}{\log{t}}$ with some constant $C>0$, and thus $\phi(t)\geq c \sqrt{\log{t}} $ with $c=C^{-1/2}>0$. Moreover it is expected that $\psi(t)=O(t^\beta)$ with some $\beta<1/2$. It is also known that $\psi(t)\geq c$ and $\phi(t)\leq C \sqrt{t}$ with some $c,\,C>0$ (See \cite[(1.13)]{Kes93} for a lattice FPP, and Appendix for the detailed proof).

 The following is our main result, which gives an explicit lower bound coming from the variance $\psi(t)$. It should be noted that there have been no results concerning with the lower bound of the non-random fluctuation in Euclidean FPP. The related work in a lattice FPP is introduced in Section~\ref{related work section}.
   \begin{thm}\label{thm-main}
     There exists $c>0$ such that for any $x\in\R^d$ satisfying $|x|>1$,
     \begin{equation*}
       \E {\rm T}(0,x) -{\rm g}|x| \geq c\log{\phi(|x|)}.
       \end{equation*}
     In particular, by Jensen's inequality,     
     \begin{equation*}
\E |{\rm T}(0,x) -{\rm g}|x|| \geq c\log{\phi(|x|)}.
     \end{equation*}
   \end{thm}
   \subsection{Related works}\label{related work section}
  The main issue of FPP is to understand the behavior of ${\rm T} (0,x)$ as $|x|\to\infty$. Since Kingman proved a kind of law of large numbers as in \eqref{kingman}, the next question was the asymptotic of ${\rm T} (0,x)-{\rm g}(x)$ with a natural scaling. Thus, the typical order of ${\rm T} (0,x)-{\rm g}(x)$ was of great interest in search of scaling.
\
 To study this, Kesten considered the following decomposition:
  \ben{\label{decomposition Kesten}
{\rm T}(0,x)-{\rm g}(x)=[{\rm T}(0,x)-\E {\rm T}(0,x)]+ [\E {\rm T}(0,x)-{\rm g}(x)].
}
The first term ${\rm T}(0,x)-\E {\rm T}(0,x)$ is called the random fluctuation, while the second term ${\rm T}(0,x)-\E {\rm T}(0,x)$ is called the non-random fluctuation. Kesten's idea is the following. First, we study the variance of ${\rm T} (0,x)$ and estimate the random fluctuation from it. Second, using the estimate of the random fluctuation, we estimate the non-random fluctuation. 

 Following the idea, there have been several attempts to study them \cite{Kes93,Alex97,ADH15,DK16, DW16, BDG18}. In particular, Alexander \cite{Alex97} developed a new and strong method to derive the upper bound of the non-random fluctuation from a concentration inequality for the random fluctuation. Nevertheless, there are few results on the lower bounds of the non-random fluctuations due to the lack of understanding and techniques. 

In the classical FPP, the author proved the divergence of the non-random fluctuation \cite{N19}. However, there are at least two drawbacks. First, the result was not stated for a fixed direction. Second, the estimate is anything but sharp, where the lower bound is given by $(\log\log{n})^{\frac{1}{d}}$.   In this paper, changing the model, we overcome these problems. Indeed, by the rotational invariance of our model, we not only prove the result for any fixed direction, but  improve the bound, though we are not sure if this is sharp. Moreover, the argument may be transparent because some of the cumbersome terms disappear in our argument.
\subsection{Notation and terminology}
This subsection collects some notations and terminologies for the proof.
\begin{itemize}
  \item   Let us define the Euclidean ball ${\rm B}(x,r)$ for $x\in\R^d$ and $r>0$ as
    $${\rm B}(x,r)=\{y\in\R^d|~|x-y|\le r\}.$$
    For $x=0$, we simply write ${\rm B}(r)$ instead of ${\rm B}(x,r)$.
  \item  For $a\in\R$, $\lf a\rf$ is the greatest integer less than or equal to $a$. Given $x=(x_i)_{i=1}^d\in\R^d$, we define $\lf x\rf=(\lf x_i\rf)_{i=1}^d.$
\item Given $a,b,y\in\R^d$, we define ${\rm T}(a,y,b)={\rm T}(a,y)+{\rm T}(y,b)$, which is the first passage time from $D(a)$ to $D(b)$ passing through $D(y)$.
  \item We denote by $\Gamma(x,y)$ and $\Gamma(x,y,z)$ the optimal paths of ${\rm T}(x,y)$ and ${\rm T}(x,y,z)$, respectively.  
\item  Given a Borel set $A\subset \R^d$, we denote by ${\rm Vol}(A)$ the $d$-dimensional volume of $A$.
\end{itemize}
\section{Proof of the main theorem}

We only consider the $\mathbf{e}_1$-direction, i.e. $x=\me_1$, since another direction is the same by the rotational invariance. We write ${\rm T}_n={\rm T}(0,n\mathbf{e}_1)$. Let us denote $\L=\{(x_i)\in\R^d|~x_1=0\}$. Given sufficiently large $n>0$, one can find a finite subset $\L_n$ of $\L$ such that 
\begin{equation}\label{Sn-cond}
  \begin{cases}
  \sharp \L_n = \lf \phi(n)^{1/2}\rf,\\
  \text{if $a\neq b\in \L_n$, }|a-b|\ge \sqrt{n}\,\phi(n)^{-1/2},\\
  \text{for any $a\in \L_n$, } \sqrt{n}\phi(n)^{-1/2}\le |a|\le \sqrt{n}.\\
  \end{cases}
\end{equation}
Given $y\in\L_n$, let us define
\ben{\label{kAy}
\mathcal{A}^y_{\eqref{kAy}}=\{\forall z\in \L_n\text{ with $z\neq y$},~{\rm T}(-n\me_1,y,n\me_1)< {\rm T}(-n\me_1,z,n\me_1)\}.
}
\begin{prop}\label{estimate1}
 For any $K>0$,
\begin{equation}\label{prop1}
  2(\E {\rm T}_n-{\rm g}n) \geq K\sum_{y\in \L_n}\P(\{{\rm T}(-n\me_1,0,n\me_1)-{\rm T}(-n\me_1,y,n\me_1)> K\}\cap\mathcal{A}_{\eqref{kAy}}^y).
\end{equation}
\end{prop}
\begin{proof}
  For any $n>1$, observe that by \eqref{time constant v.s. fpt},
  \begin{equation*}\label{Key-ineq}
    \begin{split}
     2(\E {\rm T}_n-{\rm g}n)
    &= \E [{\rm T}(-n\me_1,0,n\me_1)-{\rm T}(-n\me_1,n\me_1)]+(\E[{\rm T}(-n\me_1,n\me_1)]-2{\rm g}n)\\
      &\geq \E [{\rm T}(-n\me_1,0,n\me_1)-{\rm T}(-n\me_1,n\me_1)].
      \end{split}
    \end{equation*}
 Since ${\rm T}(x,y,z)\ge {\rm T}(x,z)$ for any $x,y,z\in\R^d$ and $\{\mathcal{A}^y_{\eqref{kAy}}\}_{y\in \L_n}$ are disjoint, we have
\begin{equation*}
\begin{split}
  &\quad \E [{\rm T}(-n\me_1,0,n\me_1)-{\rm T}(-n\me_1,n\me_1)]\\
  &\geq \sum_{y\in \L_n}\E[{\rm T}(-n\me_1,0,n\me_1)-{\rm T}(-n\me_1,n\me_1);~\mathcal{A}^y_{\eqref{kAy}}]\\
&\ge \sum_{y\in \L_n}\E[{\rm T}(-n\me_1,0,n\me_1)-{\rm T}(-n\me_1,y,n\me_1);~\mathcal{A}^y_{\eqref{kAy}}].
\end{split}
\end{equation*}
By the first moment mothods, this is further bounded from below by the RHS of \eqref{prop1}.
\end{proof}
We take $K= K_n(\theta)=\theta\,\log{\phi(n)}$ for a fixed $\theta$ to be chosen later. The next proposition is useful to estimate the right hand side of \eqref{prop1} from below.
\begin{prop}\label{estimate2}
  There exists $\theta>0$ such that for sufficiently large $n>1$ and $y\in \L_n$,
 \begin{equation}
   \begin{split}
     &\quad \P(\{{\rm T}(-n\me_1,0,n\me_1)-{\rm T}(-n\me_1,y,n\me_1)> K_n\}\cap\mathcal{A}^y_{\eqref{kAy}})\\
     &\geq \exp{\left(-\frac{1}{4}\log{\phi(n)}\right)}\left(\frac{3}{4}-K_n^{-1}(\E[{\rm T}(-n\me_1,y,n\me_1)]-2{\rm g}n)\right). \label{form:est2}
     \end{split}
  \end{equation}
\end{prop}
 Before proving  Propositions~\ref{estimate2}, we shall complete the proof of Theorem~\ref{thm-main}. First, suppose that $n$ is large enough and there exists $y\in \L_n$ such that  $\E [{\rm T}(-n\me_1,y)]- {\rm g}n\geq K_n/4.$ By $n\leq |y+n\me_1|\leq \sqrt{n^2+n}\leq n+ 1$ and the rotational invariance of ${\rm T} $,
 \begin{align*}
   \E[{\rm T}_n] -{\rm g}n&= \E \left[{\rm T}\left(0,n\fr{y+n\me_1}{|y+n\me_1|}\right)\right]-{\rm g}n \\
   &= \E [{\rm T}(0,y+n\me_1)] -{\rm g}n+ \E \left[{\rm T}\left(0,n\fr{y+n\me_1}{|y+n\me_1|}\right)-{\rm T}(0,y+n\me_1)\right]\\
   &\geq \E [{\rm T}(0,y+n\me_1)] -{\rm g}n- \E \left[{\rm T}\Le(n\fr{y+n\me_1}{|y+n\me_1|}, |y+n\me_1| \fr{y+n\me_1}{|y+n\me_1|}\Ri)\right]\\
   &=  \E [{\rm T}(-n\me_1,y)] -{\rm g}n- \E{\rm T}_{|y+n\me_1|-n}
   \geq \frac{K_n}{8}= \frac{\theta}{8}\log{\phi(n)},
   \end{align*}
  as desired.  Otherwise, if for any $y\in \L_n$, $\E [{\rm T}(-n\me_1,y)- {\rm g}n]\leq K_n/4$, then
    \begin{equation*}
      \begin{split}
        \E[{\rm T}(-n\me_1,y,n\me_1)]-2{\rm g}n&=2\E [{\rm T}(-n\me_1,y)- {\rm g}n]\\
        &\leq \frac{1}{2}K_n=\frac{\theta}{2}\log{\phi(n)}.
      \end{split}
      \end{equation*}
  This, combined with Proposition~\ref{estimate1} and \ref{estimate2}, implies that
\begin{equation*}
\begin{split}
\E[{\rm T}_n] -{\rm g}n&\ge \frac{1}{4}K_n \sum_{y\in \L_n} \exp{\left(-\fr{1}{4} \log{\phi(n)}\right)}\\
&= \frac{1}{4}K_n\lf\phi(n)^{1/2}\rf  \exp{\left(-\frac{1}{4}\log{\phi(n)}\right)}> \frac{K_n}{4}=\frac{\theta}{4}\log{\phi(n)}.
\end{split}
\end{equation*}    
 Therefore, the proof of Theorem~\ref{thm-main} is completed. Thus, it remains to prove Proposition~\ref{estimate2}. We prepare some notations for the proof.
  \begin{Def}
 We define the events $\mathcal{A}_{\eqref{kA1}}(\L_n)$, $\mathcal{A}_{\eqref{kA2}}(\L_n)$ and $\mathcal{A}_{\eqref{kA3}}(\L_n)$ as
  \aln{
\mathcal{A}_{\eqref{kA1}}(\L_n)&= \left\{\forall a,b\in \L_n\cup \{0\} \text{ with $a\neq b$, }{\rm T}(a,b)\ge \sqrt{n}\,\phi(n)^{-3/5}\right\},\label{kA1}\\
\mathcal{A}_{\eqref{kA2}}(\L_n)&= \left\{\forall y\in \L_n\cup \{0\},~\max_{z=-n\me_1,n\me_1}\{|{\rm T}(z,y)-\E[{\rm T}(z,y)]|\}\le \sqrt{n}\,\phi(n)^{-2/3}\right\}{\SN ,}\label{kA2}\\
\mathcal{A}_{\eqref{kA3}}(\L_n)&= \mathcal{A}_{\eqref{kA1}}(\L_n)\cap \mathcal{A}_{\eqref{kA2}}(\L_n).\label{kA3}
  } 
  \end{Def}
If it is clear from the context, we simply write $\mathcal{A}_{\eqref{kA1}}$,\,$\mathcal{A}_{\eqref{kA2}}$, $\mathcal{A}_{\eqref{kA3}}$ instead of $\mathcal{A}_{\eqref{kA1}}(\L_n)$, $\mathcal{A}_{\eqref{kA2}}(\L_n)$, $\mathcal{A}_{\eqref{kA3}}(\L_n)$, respectively. Let $\delta$ be  a sufficiently small positive number to be specified later, and set $C_{\delta}=4(1+\delta^{-1}).$
  \begin{Def} Recall the notation $\alpha$ from $\eqref{definition of fpt}$. We define 
  \aln{
    {\rm V}_{\eqref{defV}}&=\left\{y\in \L_n\left| \begin{array}{c}
      \forall \ell\in \Z\text{ with } \ell\geq (K_n)^{\frac{1}{2\alpha}},\, \forall x\in {\rm B}(y,C_{\delta}K_n+\ell)\cap\Z^d \\
      \text{ s.t. }\Xi\cap \B(x,\ell^{1/2})\neq \emptyset
      \end{array}\right.\right \},  \label{defV}\\
  {\rm W}_{\eqref{defW}}&= \{y\in \L_n|~\forall a,b\in {\rm B}(y,2C_{\delta} K_n)\text{ with }|a-b|\ge K_n,\,{\rm T}(a,b)\ge \delta |a-b|\}, \label{defW}\\
  {\rm X}_{\eqref{defX}}&= \{y\in \L_n|~ {\rm T}(-n\me_1,y,n\me_1)-{\rm T}(-n\me_1, n\me_1) < K_n\}, \label{defX}\\
  {\rm Y}_{\eqref{defY}}&= {\rm V}_{\eqref{defV}}\cap {\rm W}_{\eqref{defW}}\cap {\rm X}_{\eqref{defX}}.\label{defY}
  }
  \end{Def}

  \begin{prop}\label{estimate-1}
    \aln{
      &\lim_{n\to\infty}\inf_{\L_n}\P(\mathcal{A}_{\eqref{kA1}}(\L_n))=1,  \label{A2}\\
      &\lim_{n\to\infty}\inf_{\L_n}\P(\mathcal{A}_{\eqref{kA2}}(\L_n))=1, \label{A3}\\
      &\lim_{n\to\infty}\inf_{\L_n}\min_{y\in \L_n}\P(y\in {\rm V}_{\eqref{defV}}\cap {\rm W}_{\eqref{defW}})=1,\label{black:eq}
    }
      where $\L_n$ runs over all subset of $\L$ satisfying \eqref{Sn-cond}.
  \end{prop}
  We pospone the proof until Appendix, but we give some words on the proof here. In fact, \eqref{A2} comes from the linearity of the first passage time, i.e, ${\rm T} (a,b)=O(|a-b|)$ with high probability, and \eqref{A3} comes from the variance $\psi(t)={\rm Var}({\rm T} (0,t\mathbf{e}_1))$ and $\sqrt{n}\,\phi(n)^{-2/3} \gg  \sqrt{n}\,\phi(n)^{-1} =\sqrt{\psi(n)}$. On the other hand, \eqref{black:eq} comes from basic computations of the Poisson point process and the linearity of the first passage time. 

Given $y\in \L_n$,   for the optimal path $(\gamma_y(i))_{i=1}^l=\Gamma(-n\me_1,y,n\me_1)$, we set
  $$s_y=\min\{i\in\{1,\cdots,l\}|\, \gamma_y(i)\in {\rm B}(y,C_{\delta}K_n)\},~t_y=\max\{i\in\{1,\cdots,l\}|\, \gamma_y(i)\in {\rm B}(y,C_{\delta}K_n)\}.$$
  \begin{prop}\label{jump-size}
    On the event $\{y\in {\rm V}_{\eqref{defV}}\}$, 
    $$\max\{|\gamma_y(s_y)-\gamma_y(s_y-1)|,|\gamma_y(t_y)-\gamma_y(t_y+1)|\}\leq (K_n)^{\frac{1}{2\alpha}}+1.$$
  \end{prop}
  
  \begin{figure}[b]\label{figure}
  \includegraphics[width=7.3cm]{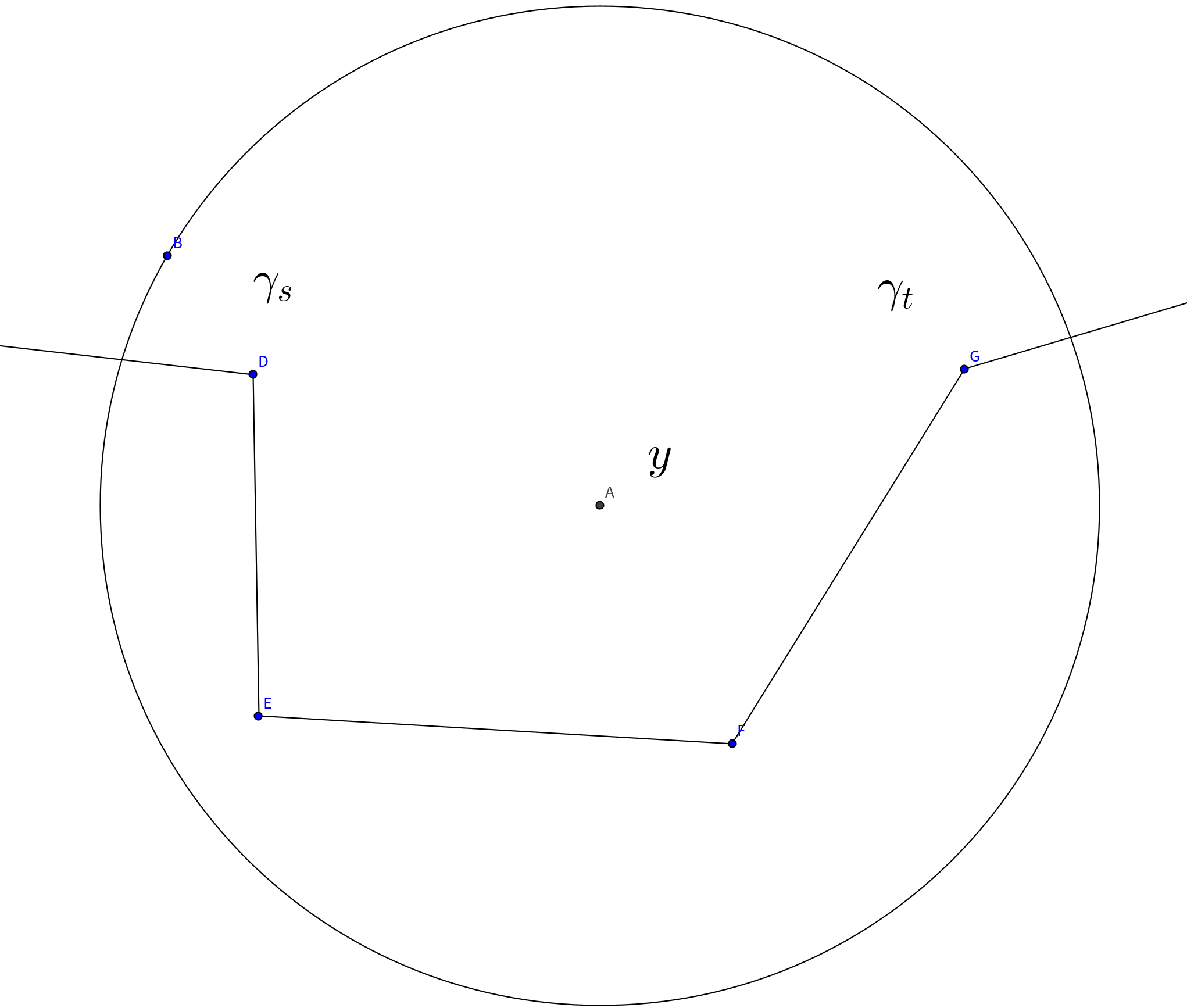}
  \includegraphics[width=7.3cm]{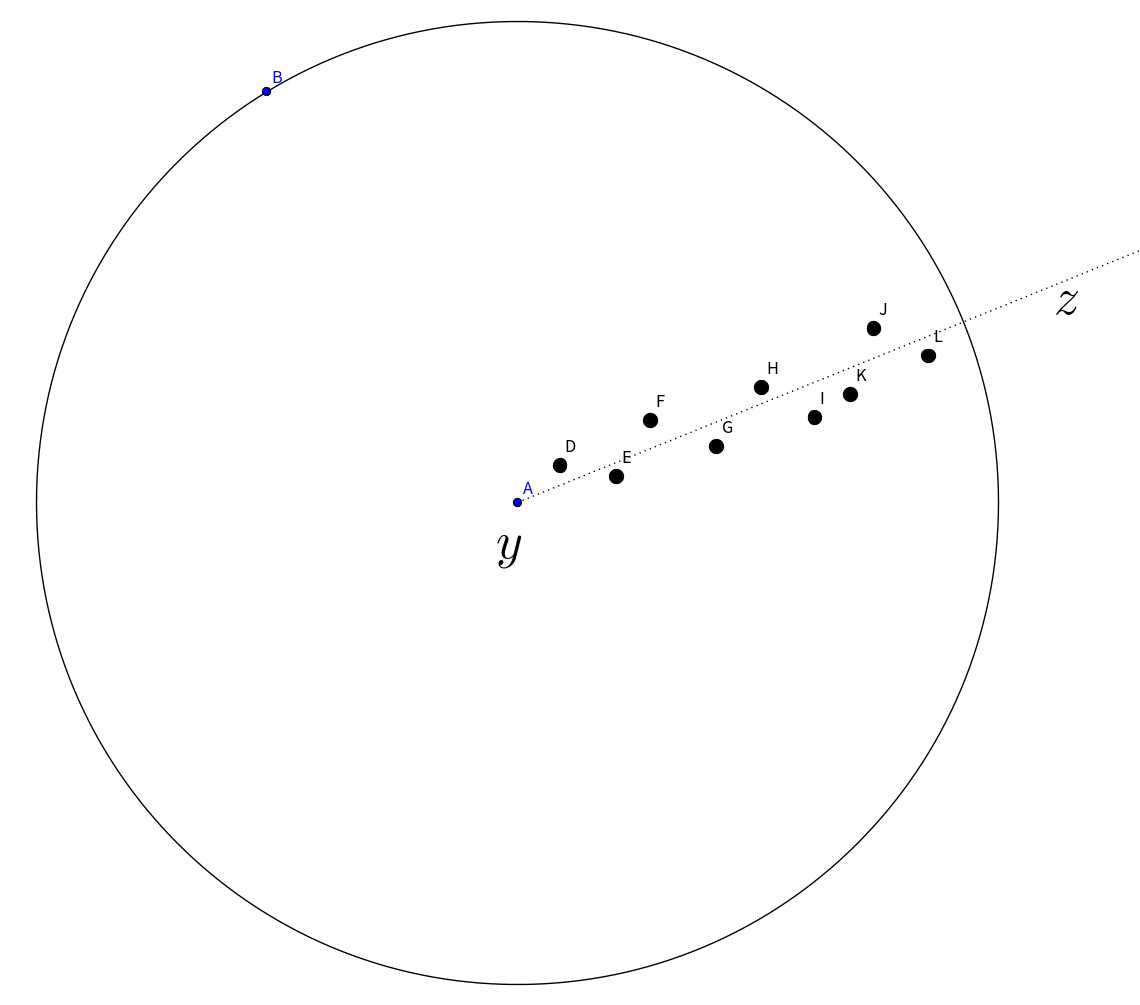}
\caption{}
Left: $\gamma_s$ and $\gamma_t$, Right: Schematic picture of $\kC_{c,y}(z)$ 
  \label{fig:one}
\end{figure}
  \begin{proof}
    For simplicity of notation, we drop subscripts $y$ in the proof such as $s=s_y$, $\gamma_i=\gamma_y(i)$.  Let $\ell=\lf |\gamma_{s}-\gamma_{s-1}|\rf$. Suppose $\ell\geq (K_n)^{\fr{1}{2\alpha}}$ and we shall derive a contradiction.\\

    Since $\left\lf\frac{\gamma_{s}+\gamma_{s-1}}{2}\right\rf\in {\rm B}(y,C_{\delta}K_n+\ell)$ and $y\in {\rm V}_{\eqref{defV}}$,
    $${\rm B}\left(\left\lf\frac{\gamma_{s}+\gamma_{s-1}}{2}\right\rf,\ell^{1/2}\right)\cap \Xi\neq \emptyset.$$
    Let us take $x\in {\rm B}(\left\lf\frac{\gamma_{s}+\gamma_{s-1}}{2}\right\rf,\ell^{1/2})\cap \Xi$. Since the jump $\{\gamma_{s-1},\gamma_{s}\}$ is itself optimal,
    \al{
      \ell^\alpha&\leq |\gamma_{s-1}-\gamma_{s}|^\alpha={\rm T}(\gamma_{s-1},\gamma_{s})\\
      &\leq |\gamma_{s-1}-x|^\alpha+|\gamma_{s}-x|^\alpha\\
      &= \left|\gamma_{s-1}-\frac{\gamma_{s}+\gamma_{s-1}}{2}+\frac{\gamma_{s}+\gamma_{s-1}}{2}-x\right|^\alpha+\left|\gamma_{s}-\frac{\gamma_{s}+\gamma_{s-1}}{2}+\frac{\gamma_{s}+\gamma_{s-1}}{2}-x\right|^\alpha\\
      &\leq 2\left( \left|\frac{\gamma_{s}+\gamma_{s-1}}{2}\right| + \left|\frac{\gamma_{s}+\gamma_{s-1}}{2}-x\right| \right)^\alpha.
    }
    Since $\left|\frac{\gamma_{s}-\gamma_{s-1}}{2}\right|\leq \frac{\ell}{2} +1$ and $\left|\frac{\gamma_{s}+\gamma_{s-1}}{2}-x\right|\leq \ell^{1/2}+d\leq 2\ell^{1/2}$ and $\ell\geq (K_n)^{\fr{1}{2\alpha}}$, for sufficiently large $n$, this is bounded from above by
    \al{
      &\quad 2\left( \frac{1}{2}+3\ell^{-1/2}\right)^\alpha \ell^\alpha<\ell^\alpha.
    }
   Therefore $\ell^\alpha<\ell^\alpha$, which is a contradiction. Thus $\ell< (K_n)^{\fr{1}{2\alpha}}$ and $|\gamma_y(s_y)-\gamma_y(s_y-1)|\leq (K_n)^{\fr{1}{2\alpha}}+1$. Similarly, we obtain $|\gamma_y(t_y)-\gamma_y(t_y+1)|\leq (K_n)^{\fr{1}{2\alpha}}+1$.
    \end{proof}
Given $z_1,z_2\in  {\rm B}(2C_{\delta}K_n)$, we define  the event 
\ben{\label{kB}
 \kB_{\eqref{kB}}(z_1,z_2)= \left\{ |\gamma_y(s_y)-(y+z_1)| \leq d,\,|\gamma_y(t_y)-(y+z_2)| \leq d\right\},
}
where $d$ is the dimension. Given $x\in \R^d$ and $c,K>0$, we define
  $$\Z_{c,K}(x)=\{k\in  \,\Z_{\geq 0}|\,\, 2c\,k|x|\leq K-1\}.$$
Given $y\in\L_n$, $z\in\R^d\backslash\{0\}$ and $c>0$, we define 
$$\kC_{c,y}(z)=\Le\{\forall k\in \Z_{c,C_{\delta}K_n}\Le(\fr{z}{|z|}\Ri),\, \Xi\cap {\rm B}\left(y+2c\,k\fr{z}{|z|},c\right)\neq \emptyset\Ri\}.$$
Roughly speaking, $\kC_{c,y}(z)$ implies that there are ubiquitous points of $\Xi$ around the line segment $\{y+tz|~t\geq 0\}\cap \rmB(y,C_{\delta}K_n)$ (See Figure~\ref{figure}). Note that, for $c<1/4$, $\kC_{c,y}(z)$ depends only on $\Xi\cap {\rm B}(y,C_{\delta}K_t)$.
 Independently of $\Xi$, we take independent random variables $Z_1,Z_2$ with uniform distributions on ${\rm B}(2C_{\delta}K_n)\cap (\Z^d\backslash \{0\})$.
\begin{lem}\label{key:lem}
 If we take $c>0$ sufficiently small such that $4^\alpha c^{\alpha-1}C_{\delta}<\frac{1}{2}$, then for sufficiently large $n>1$ and $y\in \L_n$,
  \begin{equation}
\begin{split}
&\quad\P(\{{\rm T}(-n\me_1,0,n\me_1)-{\rm T}(-n\me_1,y,n\me_1)> K_n\}\cap \mathcal{A}^y_{\eqref{kAy}})\\
  &\ge \min_{z_1,z_2\in {\rm B}(2C_{\delta}K_n)\backslash \{0\}}\P\left(\kC_{c,y}(z_1)\cap \kC_{c,y}(z_2)\right)\P\left( \mathcal{A}_{\eqref{kA3}}\cap\{y\in {\rm Y}_{\eqref{defY}}\}\cap \kB_{\eqref{kB}}(Z_1,Z_2)\right). \label{key} 
  \end{split}
\end{equation}
  \end{lem}
\begin{proof}
  We first explain the idea of the proof. We start with the event $ \mathcal{A}_{\eqref{kA3}}\cap\{y\in {\rm Y}_{\eqref{defY}}\}$. Then we resample all the configurations in ${\rm B}(y,C_{\delta}K_n)$ and suppose $\kC_{c,y}(Z_1)\cap \kC_{c,y}(Z_2)\cap  \kB_{\eqref{kB}}(Z_1,Z_2)$ after resampling. Then we will check that ${\rm T}(-n\me_1,y,n\me_1)$ decreases by at least $2K_n$. On the other hand, since $y$ and $0$ are far away from each other, ${\rm T}(-n\me_1,0,n\me_1)$ is unchanged. Similarly, we have the same thing for $\{{\rm T}(-n\me_1,z,n\me_1)\}_{z\neq y\in \L_n}$. Thus we get $\{{\rm T}(-n\me_1,0,n\me_1)-{\rm T}(-n\me_1,y,n\me_1)> K_n\}\cap \mathcal{A}^y_{\eqref{kAy}}$ after resampling. To make the above rigorous, we use the resampling argument introduced in~\cite{BK93}.\\
  
  Let $\Xi^*$ be an independent copy of the Poisson point process $\Xi$. We assume that $(\Xi,\Xi^*,Z_1,Z_2)$ are all independent. We enlarge the probability space so that we can measure the event depending on them and we still denote the joint probability measure by $\P$.   We define the resampled Poisson point process as
  $$\widetilde{\Xi}=(\Xi\cap ({\rm B}(y,C_{\delta}K_n))^c )\cup (\Xi^*\cap {\rm B}(y,C_{\delta}K_n)).$$
  We write $\widetilde{\rm T}(a,b)$ for the first passage time from $a$ to $b$ with respect to $\widetilde{\Xi}$. Similarly, we define $\widetilde{\rm T}(a,y,b)$, $\widetilde{\kC}_{c,y}(z)$  etc. Note that the distributions of $\Xi$ and $\widetilde{\Xi}$ are the same under $\mathbb{P}$ since $\Xi$ and $\Xi^*$ are independent. Thus the LHS of \eqref{key}  is equal to
  \al{
    \P(\widetilde{\mathcal{A}}^y_{\eqref{kAy}}\cap \{\widetilde{\rm T}(-n\me_1,0,n\me_1)-\widetilde{\rm T}(-n\me_1,y,n\me_1)> K_n\}),
  }
  where
  $$\widetilde{\mathcal{A}}^y_{\eqref{kAy}}=\{\forall z\in \L_n\text{ with $z\neq y$},~\widetilde{\rm T}(-n\me_1,y,n\me_1)< \widetilde{\rm T}(-n\me_1,z,n\me_1)\}.$$
 By independence of $\Xi$ and $\Xi^*$, the right hand side of \eqref{key} is bounded from above by
 \aln{
   &\quad \sum_{z_1,z_2} \P(Z_1=z_1,Z_2=z_2)\P(\widetilde{\kC}_{c,y}(z_1)\cap \widetilde{\kC}_{c,y}(z_2))\P (\mathcal{A}_{\eqref{kA3}}\cap\{y\in {\rm Y}_{\eqref{defY}}\}\cap \kB_{\eqref{kB}}(z_1,z_2))    \nonumber\\
   &= \P(\widetilde{\kC}_{c,y}(Z_1)\cap \widetilde{\kC}_{c,y}(Z_2)\cap \mathcal{A}_{\eqref{kA3}}\cap\{y\in {\rm Y}_{\eqref{defY}}\}\cap \kB_{\eqref{kB}}(Z_1,Z_2)). \label{key2}
  }
     Thus, it suffices to show that the event inside the probability in \eqref{key2} implies $\widetilde{\mathcal{A}}^y_{\eqref{kAy}}$ and $\widetilde{\rm T}(-n\me_1,0,n\me_1)-\widetilde{\rm T}(-n\me_1,y,n\me_1)> K_n$. To do this, we suppose that $(\Xi,\Xi^*,Z_1,Z_2)$ belongs to the event in \eqref{key2}.\\

      \noindent \underline{Step 1}   We prove that $\widetilde{\rm T}(-n\me_1,y,n\me_1)+ 2K_n<{\rm T}(-n\me_1,y,n\me_1)$. Take the optimal path $(\gamma_i)^{\ell}_{i=1}=\Gamma(-n\me_1,y,n\me_1)$ and let
      $$s=\min \{i\in\{1,\cdots,\ell\}|~\gamma_i\in {\rm B}(y,C_{\delta}K_n)\}\text{ and }t=\max \{i\in\{1,\cdots,\ell\}|~\gamma_i\in {\rm B}(y,C_{\delta}K_n)\}.$$
      Since $|\gamma_s-(y+Z_1)|\leq d$ where $d$ is the dimension, taking $k=\lf(2c)^{-1}(|Z_1|-2d)\rf\lor 0$,  one has $2c k \leq C_{\delta}K_n-1$.  On the event $\widetilde{\kC}_{c,y}(Z_1)$, for any $0\leq k'\leq k$, there exists $q_{k'}\in \widetilde{\Xi }\cap {\rm B}\Le(y+2c\,k' \frac{Z_1}{|Z_1|},c\Ri)$. Then, by Proposition~\ref{jump-size},
            \aln{
        |\gamma_{s-1}-q_k|&\leq  |\gamma_{s-1}-\gamma_s|+\Le|(y+Z_1)-\gamma_s\Ri|+\Le|(y+Z_1)-q_k\Ri| \nonumber\\
&\leq \Le( (K_n)^{\fr{1}{2\alpha}}+1\Ri)+d+\Le|\Le(y+2c\,k \frac{Z_1}{|Z_1|}\Ri)-(y+Z_1)\Ri|+\Le|q_k-\Le(y+2c\,k \frac{Z_1}{|Z_1|}\Ri)\Ri|\nonumber\\
        &\leq \Le( (K_n)^{\fr{1}{2\alpha}}+1\Ri)+d+3d+c\leq 2(K_n)^{\fr{1}{2\alpha}}.\label{jumpjump}
        }
      Since $\gamma_{s-1}\in {\rm B}(y,2C_{\delta}K_t)\backslash{\rm B}(y,C_{\delta}K_n)$ and $C_{\delta}=4(1+\delta^{-1})$, on the event $\kA_{\eqref{kA3}}$,
      \aln{
        {\rm T}(\gamma_{s-1},y)\geq \delta|\gamma_{s-1}-y|\geq \delta C_{\delta}K_n\geq2K_n. \label{jump2}
        }
       Furthermore, on the event $\widetilde{\kC}_{c,y}(Z_1)$,
      \aln{
      \widetilde{\rm T}(q_k,y)&\leq\sum_{i=1}^k |q_i-q_{i-1}|^\alpha\nonumber\\
        &\leq k (4c)^\alpha 
        \leq 4^\alpha c^{\alpha-1} C_{\delta} K_n.\hspace{6mm}(\text{by }  k\leq c^{-1} C_{\delta}K_n) \label{jump3}
        }
      Thus, we have
      \al{
        \widetilde{\rm T}(-n\me_1,y)&\leq \widetilde{\rm T}(-n\me_1,\gamma_{s-1})+\widetilde{\rm T}(\gamma_{s-1},y)&\\
        &\leq {\rm T}(-n\me_1,\gamma_{s-1})+|\gamma_{s-1}-q_k|^\alpha +\widetilde{\rm T}(q_k,y)&\\
        &\leq {\rm T}(-n\me_1,\gamma_{s-1})+2^\alpha K_n^{1/2} + 4^{\alpha} c^{\alpha-1} C_{\delta} K_n\hspace{6mm}&(\text{by }  \eqref{jumpjump},\,\eqref{jump3})\\
        &\leq {\rm T}(-n\me_1,\gamma_{s-1})+ K_n&\left(\text{by }  4^{\alpha} c^{\alpha-1} C_{\delta} < 2^{-1}\right)\\
        &= {\rm T}(-n\me_1,y)-{\rm T}(\gamma_{s-1},y) + K_n <{\rm T}(-n\me_1,y)-K_n. &(\text{by }  \eqref{jump2})
      }
      Similarly, $\widetilde{\rm T}(y,n\me_1)\leq {\rm T}(y,n\me_1)-K_n$ holds. Consequently, we obtain
      $$\widetilde{\rm T}(-n\me_1,y,n\me_1)< {\rm T}(-n\me_1,y,n\me_1)-2 K_n.$$

      \underline{Step 2}  We prove that $\widetilde{\rm T}(-n\me_1,y,n\me_1)+K_n< \widetilde{\rm T}(-n\me_1,z,n\me_1)$ for any $z\in \L_n\cup\{0\}\text{ with $z\neq y$}$.  Let $z\in \L_n\cup\{0\}$ with $z\neq y$. If $\widetilde{\Gamma}(-n\me_1,z,n\me_1)$ does not touch with ${\rm B}(y,C_{\delta}K_n)$, then   one has ${\rm T}(-n\me_1,z,n\me_1) \leq\tilde{{\rm T}}(-n\me_1,z,n\me_1)$ and thus
      \aln{\label{step 2:in}
        \widetilde{\rm T}(-n\me_1,y,n\me_1)&\leq {\rm T}(-n\me_1,y,n\me_1)-2K_n\\
        &\leq {\rm T}(-n\me_1,z,n\me_1)-K_n\hspace{6mm}(\text{by }  y\in {\rm X}_{\eqref{defX}})\notag\\
        &\leq \widetilde{\rm T}(-n\me_1,z,n\me_1)-K_n,\notag
      }
      which is the desired conclusion. Hereafter, we suppose that $ {\rm B}(y,C_{\delta}K_n)\cap \widetilde{\Gamma}(-n\me_1,z,n\me_1)\neq \emptyset$. For the optimal path $(\widetilde{\gamma}_i)_{i=1}^{\widetilde{\ell}}=\widetilde{\Gamma}(-n\me_1,z,n\me_1)$, we define
      $$\widetilde{s}=\min \{i\in\{1,\cdots,\widetilde{\ell}\}|~\widetilde{\gamma}_i\in {\rm B}(y,K_n)\}\text{ and }\widetilde{t}=\max \{i\in\{1,\cdots,\widetilde{\ell}\}|~\widetilde{\gamma}_i\in {\rm B}(y,C_{\delta}K_n)\}.$$
 Then,
      \aln{
        \widetilde{\rm T}(-n\me_1,z)&= \widetilde{\rm T}(-n\me_1,\widetilde{\gamma}_{\widetilde{s}-1})+  \widetilde{\rm T}(\widetilde{\gamma}_{\widetilde{s}-1},\widetilde{\gamma}_{\widetilde{t}+1})+\widetilde{\rm T}(\widetilde{\gamma}_{\widetilde{t}+1},z)\notag\\
        &\geq {\rm T}(-n\me_1,\widetilde{\gamma}_{\widetilde{s}-1})+{\rm T}(\widetilde{\gamma}_{\widetilde{t}+1},z)\notag\\
        &\geq {\rm T}(-n\me_1,y)+{\rm T}(y,z)-{\rm T}(\widetilde{\gamma}_{\widetilde{s}-1},y)-{\rm T}(y,\widetilde{\gamma}_{\widetilde{t}+1}).\label{oppab}
      }
      By $y\in {\rm V}_{\eqref{defV}}$ and the same proof as in Proposition~\ref{jump-size},
      \al{
        {\rm T}(\widetilde{\gamma}_{\widetilde{s}-1},y)&\leq|\widetilde{\gamma}_{\widetilde{s}-1}-D(y)|^\alpha\\
        &\leq \left(|\widetilde{\gamma}_{\widetilde{s}-1}-y|+|D(y)-y|\right)^\alpha\leq (4C_{\delta}K_n)^\alpha.
        }
      Similarly
      ${\rm T}(y,\widetilde{\gamma}_{\widetilde{t}+1})\leq (4C_{\delta}K_n)^\alpha$ holds.
      Furthermore,  on the event $\mathcal{A}_{\eqref{kA3}}$,
      ${\rm T}(y,z)\geq \sqrt{n}\phi(n)^{-3/5}$  holds.
     Thus, \eqref{oppab} is further bounded from below by
      \al{
        &\quad \E {\rm T}(-n\me_1,y) - \sqrt{n}\phi(n)^{-2/3}+\sqrt{n}\phi(n)^{-3/5}-2(4C_{\delta}K_n)^\alpha\\
        &\geq \E {\rm T}(-n\me_1,z)+ \frac{1}{2}\sqrt{n}\phi(n)^{-3/5}\\
        &\geq  {\rm T}(-n\me_1,z). \hspace{6mm}(\text{by }  {\rm T}(-n\me_1,z)\leq \E {\rm T}(-n\me_1,z)+ \sqrt{n}\phi(n)^{-2/3} )
      }
     Similarly, we get $\widetilde{\rm T}(z,n\me_1)\geq {\rm T}(z,n\me_1)$, which implies $\widetilde{\rm T}(-n\me_1,z,n\me_1)\geq {\rm T}(-n\me_1,z,n\me_1)$.  Then, as in \eqref{step 2:in}, we have
      \al{
        \widetilde{\rm T}(-n\me_1,y,n\me_1)\leq \widetilde{\rm T}(-n\me_1,z,n\me_1)-K_n.
      }
    The proof is completed combining these two steps.
\end{proof}
\begin{lem}
  If $\theta<2^{-8d}\,c^d C_{\delta}^{-1}$, then
  $$\min_{z} \P(\kC_{c,y}(z))\geq \exp{\left(-\frac{1}{16}\log{\phi(n)}\right)}.$$
\end{lem}
\begin{proof}
  We simply calculate
  \al{
    \P(\kC_{c,y}(z))&\geq \left(\P({\rm B}(c)\cap \Xi\neq \emptyset) \right)^{2C_{\delta}K_t}\\
   &= \exp{(-2C_{\delta}K_t\, {\rm Vol}({\rm B}(c)))}\\
    &\geq \exp{\left(-\frac{1}{16}\log{\phi(n)}\right)}. \hspace{6mm} (\text{by }  {\rm Vol}({\rm B}(c))\leq (2c)^d)
    }
  \end{proof}
\begin{proof}[Proof of Proposition~\ref{estimate2}]
   By FKG inequality, we will compute the right hand side of \eqref{key} as 
\aln{
  &\quad \min_{z_1,z_2} \P(\kC_{c,y}(z_1)\cap \kC_{c,y}(z_2))\P(\mathcal{A}_{\eqref{kA3}}\cap\{y\in {\rm Y}_{\eqref{defY}}\}\cap \kB_{\eqref{kB}}(Z_1,Z_2))\label{elminate}\\
  &\geq  \min_{z} \P(\kC_{c,y}(z))^2 \sum_{z_1,z_2} \P(Z_1=z_1,Z_2=z_2)\P(\mathcal{A}_{\eqref{kA3}}\cap\{y\in {\rm Y}_{\eqref{defY}}\}\cap \kB_{\eqref{kB}}(z_1,z_2)).\notag
  }
  Under $\mathcal{A}_{\eqref{kA3}}\cap\{y\in {\rm Y}_{\eqref{defY}}\}$, taking $z_1=\lf \gamma_y(s_y)\rf$ and $z_2=\lf \gamma_y(t_y)\rf$, $\kB_{\eqref{kB}}(z_1,z_2)$ holds and $z_1,z_2\in  {\rm B}(2C_{\delta}K_n)\cap (\Z^d\backslash\{0\})$.  Thus
  \al{
    &\quad\sum_{z_1,z_2\in  {\rm B}(2C_{\delta}K_n)\cap \Z^d} \P(\mathcal{A}_{\eqref{kA3}}\cap\{y\in {\rm Y}_{\eqref{defY}}\}\cap \kB_{\eqref{kB}}(z_1,z_2))\\
    &= \E[\sharp\{(z_1,z_2)|~z_1,\,z_2\in  {\rm B}(y,C_{\delta}K_n)\cap \Z^d,\,\kB_{\eqref{kB}}(z_1,z_2)\};~\mathcal{A}_{\eqref{kA3}}\cap\{y\in {\rm Y}_{\eqref{defY}}\}]\\
    &\geq \P(\mathcal{A}_{\eqref{kA3}}\cap\{y\in {\rm Y}_{\eqref{defY}}\}),
  }
  
  Since $\P(Z_1=z_1,Z_2=z_2)\geq \left(\sharp [{\rm B}(2C_{\delta}K_n)\cap \Z^d] \right)^{-2}\geq (4C_{\delta}K_n)^{-2d}$,  \eqref{A2}, \eqref{A3} and $ K_n=\theta\log{\phi(n)}$, for sufficiently large $n$, \eqref{elminate} is further bounded from below by
  \al{
  &\quad \min_{z} \P(\kC_{c,y}(z))^2 (4C_{\delta}K_n)^{-2d} \P(\mathcal{A}_{\eqref{kA3}}\cap\{y\in {\rm Y}_{\eqref{defY}}\})\notag\\
&\ge \exp{\left(-\frac{1}{8}\log{\phi(n)}\right)}(4C_{\delta}\theta \log{\phi(n)} )^{-2d} \left(\P(y\in {\rm X}_{\eqref{defX}} )-\P(\mathcal{A}_{\eqref{kA3}}(\L_n)^c)- \P(y\notin {\rm V}_{\eqref{defV}}\cup {\rm W}_{\eqref{defW}})\right)\notag\\
&\ge \exp{\left(-\frac{1}{4}\log{\phi(n)}\right)}(\P({\rm T}(-n \me_1,y, n \me_1) -{\rm T}(-n \me_1, n \me_1)\le K_n)-1/4). 
}
By the first moment method, ${\rm T}(-n\me_1 ,y, n \me_1) \geq {\rm T}(-n\me_1, n \me_1)$ and  \eqref{time constant v.s. fpt},

\al{
  \P({\rm T}(-n\me_1 ,y, n \me_1) -{\rm T}(-n\me_1, n \me_1)\le K_n)&= 1 -\P({\rm T}(-n\me_1 ,y, n \me_1) -{\rm T}(-n\me_1, n \me_1)> K_n)\\
  &\geq 1-K_n^{-1}\E\left[{\rm T}(-n\me_1,y, n \me_1) -{\rm T}(-n\me_1, n \me_1)\right]\\
  &\geq 1-K_n^{-1}\E\left[{\rm T}(-n\me_1,y, n \me_1) -2 {\rm g}n \right].
}
With these observations, the proof is completed.
\end{proof}
\section{Appendix}
\subsection{Proof of $\liminf_{n\geq 0} \psi(n)>0$}
In this section, we give a proof of 
\ben{\label{lower bound of variance}
{\rm Var}({\rm T}_n))>c,
}
with some $c>0$ independent of $n$, where we recall that ${\rm T}_n={\rm T}(0,n\mathbf{e}_n)$. Let $\Xi^*$ be an independent copy of $\Xi$ and set $\widetilde{\Xi}=(\Xi\backslash \rmB(2))\cup(\Xi^*\cap \rmB(2))$. Let us denote by $\mathcal{F}$ the $\sigma$-field generated by  $\Xi\backslash \rmB(2)$. We denote by $\widetilde{\rmT}_n$ the first passage time from $0$ to $\mathbf{e}_n$ with respect to $\widetilde{\Xi}$. Then, since martingale differences are uncorrelated, we have
\al{
{\rm Var}({\rm T}_n))&= \E[(\E[\rmT_n|\kF]-\E[\rmT_n])^2]+\E[(\rmT_n-\E[\rmT_n|\kF])^2]\\
&\geq \E[\E[(\rmT_n-\E[\rmT_n|\kF])^2|\kF]]=\E[\E[(\rmT_n-\widetilde{\rmT}_n)^2|\kF]],
}
Hence, it suffice to show that there exists a non-random constant $c>0$ such that for any $n>2$
$$\E[(\rmT_n-\widetilde{\rmT}_n)^2|\kF]>c\qquad\text{a.s.}$$
To this end, we consider the event:
\al{
\kE=\{\Xi\cap \rmB(2)=\emptyset,\,\Xi^*\cap \rmB(1)\neq\emptyset,~\Xi^*\cap (\rmB(2)\backslash \rmB(1))=\emptyset\}.
}
Note that $\kE$ is independent of $\kF$. Since, on the event $\kE$, $D(0)$ (the closest point of $\Xi$ to $0$) is located in $\rmB(2)^c$,  $\widetilde{D}(0)$ (the closest point of $\widetilde{\Xi}$ to $0$) is located in $\rmB(1)$ and $\Xi\backslash \rmB(2)^c=\widetilde{\Xi}\backslash \rmB(2)^c$, we have $\widetilde{\rmT}_n-\rmT_n\geq 1$.  Hence,
\al{
\E[(\rmT_n-\widetilde{\rmT}_n)^2|\kF]&\geq \E[(\rmT_n-\widetilde{\rmT}_n)^2\mathbf{1}(\kE)|\kF]\\
&\geq \E[\mathbf{1}(\kE)|\kF]= \P(\kE),
}
and since $\kE$ is independent of $n$, the proof of \eqref{lower bound of variance} is completed.
\subsection{Proof of Proposition~\ref{estimate-1}}
\begin{proof}[Proof of \eqref{A2}]
  Note that for $a\neq b\in \L_n\cup\{0\}$,
  $$|a-b|\geq n^{1/2}\phi(n)^{-1/2}\gg \sqrt{n}\phi(n)^{-2/3}.$$ By using \cite[Lemma~1]{HN97} and $\phi(n)\leq C\, n^{1/2}$ with some $C>0$,  
  \al{
    \P((\kA_{\eqref{kA1}})^c)&\leq \sum_{a,b\in\L_n\cup\{0\}}\P({\rm T}(a,b)< \sqrt{n}\phi(n)^{-2/3}) \mathbf{1}_{\{a\neq b\}}\\
    &\leq 2\phi(n) \exp{\left(-\e \left(\sqrt{n} \phi(n)^{-1/2}\right)^\kappa\right)}\leq \exp{\left(-\e' \, n^{\kappa/4}\right)},
  }
  with some $\kappa,\e,\e'>0$
  \end{proof}
\begin{proof}[Proof of \eqref{A3}]
  Since for $y\in \L_n\cup \{0\}$, $n\leq |y|\leq  n+1$,  by  Chebyshev's inequality,
  \begin{align*}
   &\quad\P(|{\rm T}(0,y)-\E \,{\rm T}(0,y)|\ge  \sqrt{n}\,\phi(n)^{-2/3})=\P(|{\rm T}_{|y|}-\E{\rm T}_{|y|}|\ge  \sqrt{n}\,\phi(n)^{-2/3})\\
    &\leq  \P\left(|{\rm T}_n-\E {\rm T}_n|+{\rm T}(n\mathbf{e}_1,|y|\mathbf{e}_1)+\E {\rm T}(n\mathbf{e}_1,|y|\mathbf{e}_1) \ge  \sqrt{n}\,\phi(n)^{-2/3}\right)\\
    &\leq \P\left(|{\rm T}_n-\E {\rm T}_n| \ge  \frac{1}{3}\sqrt{n}\,\phi(n)^{-2/3}\right)+\P\left({\rm T}_{|y|-n}\geq \frac{1}{3}\sqrt{n}\,\phi(n)^{-2/3}\right)\\
    &\leq\frac{9\phi(n)^{4/3}}{n} \left(\E {\rm T}_{|y|-n}^2+{\rm Var}({\rm T}_n)\right)\leq C \phi(n)^{-2/3}
  \end{align*}
  with some constant $C>0$ independent of $y$ and $n$.
  Then by the union bound, we have
  \al{
      \P(\kA_{\eqref{kA2}}^c)&=\P(\exists y\in \L_{n}\cup\{0\}\text{ such that }\max_{z=-n\me_1,n\me_1}\{|{\rm T}(z,y)-\E\,{\rm T}(z,y)|\}\ge \sqrt{n}\,\phi(n)^{-2/3})\\
      &\leq 2\sharp \L_{n}\sup_{y\in  \L_{n}\cup\{0\}} \P(|{\rm T}(0,y)-\E\,{\rm T}(0,y)|\ge \sqrt{n}\,\phi(n)^{-2/3})\\
      &\leq 2C \phi(n)^{1/2}\,\phi(n)^{-2/3}=2C \phi(n)^{-1/6}.
     }
\end{proof}
\begin{proof}[Proof of \eqref{black:eq}]
  We first prove that
  $$\lim_{n\to\infty}\inf_{\L_n}\P(y\in {\rm V}_{\eqref{defV}})=1.$$
    Indeed, by the union bound,
    \al{
    \P(\{y\in {\rm V}_{\eqref{defV}}\}^c)&\leq \sum_{\ell\geq (K_n)^{\fr{1}{2\alpha}}}\sum_{x\in{\rm B}(y,C_{\delta}K_n+\ell)\cap\Z^d}\P(\Xi\cap {\rm B}(x,\ell^{1/2})=\emptyset)\\
    &= \sum_{\ell\geq (K_n)^{\fr{1}{2\alpha}}}\sharp({\rm B}(y,C_{\delta}K_n+\ell)\cap\Z^d)\,\P(\Xi\cap {\rm B}(\ell^{1/2})=\emptyset)\\
    &\leq 2^d \sum_{\ell\geq (K_n)^{\fr{1}{2\alpha}}} (K_n+\ell)^d \exp{(-\rm{Vol}({\rm B}(\ell^{1/2})))} \to 0,\text{ \hspace{3mm} as $n\to\infty$.}
    }
    It remains to prove
    $$\P(\{\text{$\exists a,b\in {\rm B}(2C_{\delta} K_n)$ s.t. $|a-b|\ge K_n$ and ${\rm T}(a,b)\geq \delta|a-b|$}\}\cap \{y\in {\rm V}_{\eqref{defV}}\})\to 0.$$
    First we note that by \cite[Lemma~1]{HN97}, for sufficiently small $\delta$,  
    \al{
      &\quad \P(\text{$\exists a,b\in {\rm B}(y,4C_{\delta} K_n)\cap \Z^d$ such that $|a-b|\ge K_n/2$ and }{\rm T}(a,b)\leq 4\delta|a-b|)\\
      &\leq \sum_{a,b\in {\rm B}(y,4C_{\delta} K_n)\cap \Z^d} \P({\rm T}(a,b)\leq 4\delta|a-b|)\mathbf{1}_{\{|a-b|\ge K_n/2\}}\nonumber\\
      &\leq (8C_{\delta} K_n)^{2d} \exp{(-\e K_n^\kappa)}\leq \exp{\Le(-\frac{\e}{2} K_n^\kappa\Ri)},\nonumber
    }
    with some $\e,\kappa>0$.  Hereafter, we suppose that $y\in {\rm V}_{\eqref{defV}}$ and for any $a,b\in {\rm B}(y,4C_{\delta} K_n)\cap \Z^d$ with $|a-b|\ge K_n/2$, ${\rm T}(a,b)> 4\delta|a-b|$. Let $a,b\in {\rm B}(y,2C_{\delta} K_n)$ with $|a-b|\geq K_n$. Then by $y\in {\rm V}_{\eqref{defV}}$,
    $$\max\{|D(a)-a|,|D(b)-b|\}\leq 2 (K_n)^{\fr{1}{2\alpha}}.$$
    Similarly,  $\max\{|D(\lf a\rf)-\lf a\rf|,|D(\lf b\rf)-\lf b\rf|\}\leq 2 (K_n)^{\fr{1}{2\alpha}}$. Hence,
    \ben{\label{concest}
      \max\{|D(a)-D(\lf a \rf)|,|D(b)-D(\lf b\rf)|\}\leq 6 (K_n)^{\fr{1}{2\alpha}}.
    }
    Since $\lf a\rf,\lf b \rf\in  {\rm B}(y,4C K_n)\cap  \Z^d$ and $|\lf a\rf-\lf b\rf|\ge |a-b|/2\geq K_n/2$,
    \al{
      {\rm T}(a,b)&={\rm T}(D(a),D(b))\\
      &\geq {\rm T} (D(\lf a\rf),D(\lf b\rf ))-{\rm T} (D(a),D(\lf a \rf))- {\rm T} (D(b),D(\lf b \rf))\\
      &\geq {\rm T}(D(\lf a\rf),D(\lf b\rf ))-|D(a)-D(\lf a \rf)|^\alpha - |D(b)-D(\lf b \rf)|^\alpha\\
     &\geq 2\delta|a-b|  - 12^\alpha (K_n)^{\fr{1}{2}}\geq \delta|a-b|.\hspace{6mm}(\text{by }  \eqref{concest})
      }
   Therefore, the proof is completed.
    \end{proof}



\begin{thebibliography}{99}

  \bibitem{Alex97}
  Kenneth S. Alexander. \newblock Approximation of subadditive functions and convergence rates in limiting-shape results. 
\newblock {\em Ann. Probab.} 25, 30--55, 1997. \MR{1428498}
  
  \bibitem{ADH}
  A. Auffinger, M. Damron, J. Hanson,
\newblock 50 years of first-passage percolation. 
\newblock {\em University Lecture Series, 68. American Mathematical Society, Providence, RI}, 2017. \MR{3729447}

\bibitem{ADH15}
   A. Auffinger, M. Damron, and J. Hanson. \newblock Rate of convergence of the mean for sub-additive ergodic sequences. \MR{3406498}


\bibitem{BDG18}
  M. Bernstein, M. Damron, T. Greenwood
  \newblock Sublinear variance in Euclidean first-passage percolation. {\it preprint}

  
  \bibitem{BK93}
J. van den Berg and H. Kesten. \newblock Inequalities for the time constant in
first-passage percolation.
\newblock {\em Ann. Appl. Probab.} 56--80, 1993. \MR{1202515}


\bibitem[CC18]{CC18}
\textsc{F. Comets} and \textsc{C. Cosco}, \newblock{Brownian Polymers in Poissonian Environment: a survey}, \newblock{arXiv:1805.10899.}  (2018)


\bibitem{DK16}
M. Damron and N. Kubota. \newblock Rate of convergence in first-passage percolation under low moments.
\newblock {\em Stochastic Process. Appl.} 126 (10), 3065--3076, 2016

 \bibitem{DW16}
M. Damron and X. Wang. \newblock Entropy reduction in Euclidean first-passage percolation.
\newblock {\em Electron. J. Probab.} 21 (65), 1--23, 2016. \MR{3580031}

\bibitem{HN97}
C.~D. Howard and C.~M. Newman.
\newblock Euclidean models of First-passage percolation.
\newblock {\em Probability Theory and Related Fields}, 108, 153 -- 170 1997.

  \bibitem{Kes93}
  Harry Kesten. \newblock On the speed of convergence in first-passage percolation.
{\em Ann. Appl. Probab.} 3, 296-338, 1993.  \MR{1221154}
  
\bibitem{N19} 
  Shuta Nakajima.
  \newblock Divergence of non-random fluctuation in First Passage Percolation
  \newblock {\em Electron. Commun. Probab.} 24 (65), 1--13. 2019. \MR{4029434}


\end{thebibliography}
\end{document}